%% file: main.tex
\documentclass[11pt,reqno]{amsart}
\input{library}

\numberwithin{equation}{section}

\title[Fractal uncertainty principle over $\mathbb Q_p$]
{Fractal uncertainty principle over $\mathbb Q_p$}
\author[V. Badalucco]{Valentino Badalucco}
\address{
    Département de Mathématiques et Applications, École Normale Supérieure, Université PSL, 45, rue d'Ulm 75005 Paris, France
}
\email{valentino.badalucco@ens.fr}
\author[L. Franchi]{Leonardo Franchi}
\address{
Department of Pure Mathematics and Mathematical Statistics,
Centre for Mathematical Sciences,
Wilberforce Road,
Cambridge CB3 0WA,
United Kingdom
}
\email{lf511@cam.ac.uk}

\begin{document}

\begin{abstract}
We prove a fractal uncertainty principle over $\mathbb Q_p$ for porous sets, resolving a conjecture of Cohen.
\end{abstract}

\maketitle
\vspace{-0.8cm}

\input{sections/introduction}

\input{sections/reduction-to-finite}
\input{sections/proof-of-finite}

\printbibliography

\end{document}

%% file: library.tex
\usepackage{amsfonts}
\usepackage{amsmath}
\usepackage{amssymb, esint}
\usepackage{amsthm,color}
\usepackage{enumerate}
\usepackage{mathtools}
\usepackage{enumitem}
\usepackage{url}
\usepackage[hidelinks, bookmarksdepth=3]{hyperref}
\usepackage{xcolor}
\usepackage{cancel}
\usepackage{ulem}
\usepackage{caption,subcaption}
\usepackage{tikz}
\usepackage{tikz-3dplot}
\usepackage[T1]{fontenc}
\usepackage{mathrsfs}
\usepackage{commath}

\usepackage{longtable}
\newcommand{\F}{\mathbb{F}}
\usepackage{soul}
\usepackage[
style=alphabetic,
backend=biber,
maxbibnames=99,
minbibnames=99,
maxcitenames=99,
mincitenames=99
]{biblatex}
\AtBeginDocument{
  \setlength{\abovedisplayskip}{9pt plus 2pt minus 3pt}
  \setlength{\belowdisplayskip}{9pt plus 2pt minus 3pt}
  \setlength{\abovedisplayshortskip}{6pt plus 2pt minus 2pt}
  \setlength{\belowdisplayshortskip}{6pt plus 2pt minus 2pt}
  \setlength{\jot}{4pt}
}

\DeclareFieldFormat[article]{journaltitle}{#1}
\DeclareFieldFormat[book]{title}{#1}

\tdplotsetmaincoords{70}{110}

\tdplotsetmaincoords{70}{110}

\usepackage[a4paper,margin=2.8cm]{geometry}

\newcommand{\Z}{{\mathbb Z}}

\def\0{{\mathbf 0}}

\newcommand{\supp}{\operatorname{supp}}

\renewcommand{\F}{\mathcal{F}}

\newcommand{\defeq}{\vcentcolon=}

\newcommand{\one}{\mathbf{1}}

\theoremstyle{plain}
\newtheorem{thm}{Theorem}[section]

\newtheorem{cor}[thm]{Corollary}
\newtheorem{lem}[thm]{Lemma}
\newtheorem{prop}[thm]{Proposition}

\newcommand{\thistheoremnames}{}
\newtheorem*{genericthms}{\thistheoremnames}
\newenvironment{para*}[1]
  {\renewcommand{\thistheoremnames}{#1}%
   \begin{genericthms}}
  {\end{genericthms}}

\newtheorem{defn}[thm]{Definition}

\theoremstyle{remark}

\newtheorem*{claim*}{Claim}

%% file: sections/introduction.tex
\section{Introduction}
A fractal uncertainty principle says, roughly, that a function and its Fourier
transform cannot both be concentrated near fractal sets.  Since the work of
Bourgain and Dyatlov~\cite{BourgainDyatlov2018}, this principle has become a
basic tool in quantum chaos and related problems.  In particular, it has been
used to prove lower bounds for the mass of eigenfunctions on hyperbolic
surfaces~\cite{DyatlovJin2018,DyatlovJinNonnenmacher2022}, exponential decay
for damped waves~\cite{Jin2020}, and spectral gaps for open quantum
systems~\cite{DyatlovZahl2016,DyatlovZworski2020}; see also the survey
\cite{Dyatlov2019}.

An interesting feature of the one-dimensional Bourgain--Dyatlov theorem is that a
porosity condition is sufficient to obtain a fractal uncertainty principle.
Recall that if $(\Omega,d)$ is a metric space, then a set
$E\subset\Omega$ is called $\nu$-porous between the scales $\alpha_0$ and
$\alpha_1$ if, for every ball $B$ of diameter
$R\in(\alpha_0,\alpha_1)$, there exists $x\in B$ such that
$B(x,\nu R)\cap E=\varnothing.$  With this terminology, the Bourgain--Dyatlov theorem
takes the following form.

\begin{thm}[Bourgain--Dyatlov]\label{thm:bd-real-fup}
Let $\nu>0$.  There exist constants $C=C(\nu)<\infty$ and
$\beta=\beta(\nu)>0$ such that the following holds.  Let $0<h<1$.  Suppose
that $X\subset[-1,1]$ is $\nu$-porous from scales $h$ to $1$, and that
$Y\subset[-h^{-1},h^{-1}]$ is $\nu$-porous from scales $1$ to $h^{-1}$.
Then, for every $f\in L^2(\mathbb R)$,
$$
    \supp\widehat f\subseteq Y
    \quad\Longrightarrow\quad
    \|\one_Xf\|_{L^2(\mathbb R)}
    \le Ch^\beta\|f\|_{L^2(\mathbb R)}.
$$
\end{thm} 
The situation in higher dimensions is more delicate. Ordinary porosity is no longer sufficient, since it does not exclude sets containing entire directions; see also \cite{CohenDiscrete2025} for the corresponding obstruction in the discrete two-dimensional setting. Han and Schlag proved a higher-dimensional Bourgain-Dyatlov theorem when the Fourier-side set has product structure~\cite{HanSchlag2020}.  More recently, Cohen obtained a higher-dimensional result under the stronger assumption of line porosity, using a higher-dimensional Beurling--Malliavin multiplier theorem~\cite{CohenHigherDimFUP}.

The $p$-adic problem considered below has a different obstruction.  It comes instead from the failure over $\mathbb Q_p$ of the real-variable mechanism behind unique continuation.

To state our main theorem, we first introduce the relevant notation and fix our Fourier-analytic normalizations.
We write $|x|_p$ for the usual $p$-adic absolute value, normalized by $|p|_p=p^{-1}$, and, for $r>0$, we let
$B_r=\{x\in\mathbb Q_p:|x|_p\leq r\}$.
In particular, $B_1=\mathbb Z_p$. We use a Haar measure $dx$ on $\mathbb Q_p$, normalized so that $\mathbb Z_p$ has measure one.

Every non-zero $x\in\mathbb Q_p$ has a unique expansion
$$
x=\sum_{j\geq j_0}x_jp^j,
$$
where $j_0\in\mathbb Z$ and $x_j\in\{0,\ldots,p-1\}$. We use the additive character
$$
\chi(x)=\exp\left(2\pi i\sum_{\substack{j<0}}x_jp^j\right).
$$
The sum in the exponent is finite, and $\chi$ is trivial on $\mathbb Z_p$. For $\xi\in \mathbb{Q}_p$, we define the Fourier transform by
$$
\widehat f(\xi)=\int_{\mathbb Q_p}\chi(-x\xi)f(x)dx.
$$
With these normalizations, Fourier inversion and Parseval's identity take the form
$$
f(x)=\int_{\mathbb Q_p}\chi(x\xi)\widehat f(\xi)\,d\xi
$$
and
$$
\int_{\mathbb Q_p}f(x)\overline{g(x)}\,dx
=
\int_{\mathbb Q_p}\widehat f(\xi)\overline{\widehat g(\xi)}\,d\xi.
$$
In particular, the Fourier transform is unitary on $L^2(\mathbb Q_p)$, and
$$
\widehat{\mathbf 1_{\mathbb Z_p}}
=
\mathbf 1_{\mathbb Z_p}.
$$ This already marks a clear difference from the Archimedean setting. Over $\mathbb{R}$, compact Fourier support implies, by the Fourier inversion formula, that $f$ extends to an entire function, and hence a non-zero such function cannot vanish on an interval. This unique continuation phenomenon is a crucial ingredient in the Bourgain-Dyatlov argument. Over $\mathbb{Q}_p$, by contrast, the identity above exhibits a non-zero function which is both compactly supported and has compactly supported Fourier transform. This failure of unique continuation is precisely the obstruction identified by Cohen in his formulation of the $p$-adic fractal uncertainty problem~\cite{Cohen}.

We prove the following $p$-adic fractal uncertainty principle, which resolves a conjecture of Cohen.

\begin{thm}[Fractal uncertainty over $\mathbb{Q}_p$]\label{thm:padic-fup}
Let $p$ be a prime and let $0<\nu<1$.  Then there exist constants
$C=C(p,\nu)<\infty$ and $\beta=\beta(p,\nu)>0$ such that the following holds.
Let $h=p^{-n}$ with $n\ge1$.  Suppose that $X\subset B_1$ is $\nu$-porous
from scales $h$ to $1$, and that $Y\subset B_{h^{-1}}$ is $\nu$-porous from
scales $1$ to $h^{-1}$.  Then, for every $f\in L^2(\mathbb{Q}_p)$ such that
$\supp\widehat f\subseteq Y$, one has
$$
    \|\one_Xf\|_{L^2(\mathbb{Q}_p)}
    \le Ch^\beta\|f\|_{L^2(\mathbb{Q}_p)}.
$$
\end{thm}
The proof of Theorem~\ref{thm:padic-fup} passes through a finite fractal
uncertainty principle on the cyclic groups $\mathbb Z/p^n\mathbb Z$; this
finite formulation is also discussed by Cohen. 
We state this result now and prove it implies the main theorem in Section~2.
For each $n\geq1$, let
$$
G_n\defeq \mathbb Z/p^n\mathbb Z.
$$
If $0\leq j\leq n$ and $a\in\mathbb Z/p^j\mathbb Z$, we write
$$
B_j(a)\defeq \{x\in G_n:x\equiv a\pmod {p^j}\}.
$$

\begin{defn}\label{def:L-porous}
Let $L\geq1$.  A set $E\subset G_n$ is called $L$-porous if, whenever
$0\leq j\leq n-L$, every residue class modulo $p^j$ contains a residue
class modulo $p^{j+L}$ that is disjoint from $E$.
\end{defn}

Equivalently, for every $0\leq j\leq n-L$ and every
$a\in\mathbb Z/p^j\mathbb Z$, there is some
$b\in\mathbb Z/p^{j+L}\mathbb Z$ such that $ b\equiv a\pmod{p^j} $ and $B_{j+L}(b)\cap E=\varnothing $.

We write $\mathcal F_n$ for the unitary Fourier transform on $G_n$, defined
by
$$
\mathcal F_nf(\xi) \defeq p^{-n/2}\sum_{x\in G_n}
f(x)e^{-\frac{2\pi i x\xi}{p^n}}.
$$
The finite result we shall prove is the following.

\begin{thm}[Finite $p$-adic fractal uncertainty principle]
\label{thm:finite-fup}
Let $p$ be a prime and let $L\geq1$.  There exist constants
$C=C(p,L)<\infty$ and $\beta=\beta(p,L)>0$ such that, for every $n\geq1$
and every pair of $L$-porous sets $X,Y\subset G_n$, one has
$$
\|\one_Y\mathcal F_n \one_X\|_{\ell^2(G_n)\to\ell^2(G_n)}
\leq Cp^{-\beta n}.
$$
\end{thm}

To illustrate the hypothesis, let $\mathcal A$ be a proper subset of
$\{0,\ldots,p-1\}$, and let $X\subset G_n$ consist of those elements
$$
x=x_0+x_1p+\cdots+x_{n-1}p^{n-1}
$$
for which $x_j\in\mathcal A$ for every $j$.  Then $X$ is $1$-porous: after
fixing any initial string of digits, one may choose the next digit outside
$\mathcal A$.  Choosing $X$ and $Y$ in this way from two proper digit sets
gives the arithmetic Cantor sets considered by Dyatlov and
Jin~\cite{DyatlovJin2017}.  Theorem~\ref{thm:finite-fup}, however, requires
no self-similarity.  The missing descendant may vary from one residue class
to another and from one scale to the next.

Let us briefly describe the proof of Theorem~\ref{thm:finite-fup}. For each
$n\geq1$, let $R_n$ be the largest restricted Fourier norm between two
$L$-porous subsets of $\mathbb Z/p^n\mathbb Z$. Separating the high and low
$p$-adic digits gives the submultiplicative estimate
$$
R_{m+n}\leq R_mR_n.
$$
It is therefore enough to prove that $R_N<1$ at some single scale
$N=N(p,L)$. This strict inequality follows from a result of Osgood,
Siripuram and Wu on universal sampling sets~\cite{OSW}, which plays here
the role of a finite substitute for the analytic continuation principle
available in the Archimedean setting. Submultiplicativity then amplifies
this fixed-scale improvement into the required exponential decay.

After this work was made public, we learned that Ruxi Shi had independently obtained a proof of the $p$-adic fractal uncertainty principle proved here.

\subsection*{Statement of AI use}
The proof presented in this paper was discovered and written entirely by the authors. Motivated by comparisons between classical uncertainty principles in the Euclidean and $p$-adic settings, we were led to ask whether the failure of unique continuation could be replaced by a finite injectivity statement, in which a function with suitably restricted Fourier support is determined by its values on an appropriate sampling set. Since related questions might have appeared, under different terminology, in either mathematics or engineering, we used deep search tools for academic databases that included a non-reasoning LLM component to scan each article. This search led us to the work of Osgood, Siripuram and Wu~\cite{OSW}, which provides the main external input to our proof.

\subsection*{Acknowledgments}
The authors would like to thank Alex Cohen for interesting discussions and perspectives on the topic. The second author acknowledges support from the Isaac Newton Trust through a Trinity Cambridge Research Studentship and would also like to thank Timothy Gowers for his guidance and support. The second author also thanks Cédric Pilatte for useful discussions on non-backtracking matrices, although the final argument ultimately followed a much simpler route.

%% file: sections/reduction-to-finite.tex
\section{Reduction to the finite $p$-adic fractal uncertainty principle}

In this section we show how Theorem~\ref{thm:padic-fup} follows from
Theorem~\ref{thm:finite-fup}.

\begin{prop}
Theorem~\ref{thm:finite-fup} implies Theorem~\ref{thm:padic-fup}.
\end{prop}
\begin{proof}
Fix $0<\nu<1$, and choose $L\geq 1$ such that $\nu\geq p^{1-L}$. Let $h=p^{-n}$, and let $X\subset B_1$ and $Y\subset B_{h^{-1}}$ satisfy the assumptions of Theorem~\ref{thm:padic-fup}.

Let $X_h=X+B_h$ and $Y_1=Y+B_1$. Thus, $X_h$ is the union of the balls of radius $h$ which meet $X$, and $Y_1$ is the union of the balls of radius $1$ which meet $Y$.

Let $f\in L^2(\mathbb Q_p)$ satisfy $\supp\widehat f\subseteq Y$, and put $g=\one_{B_1}f$. Since $X\subset B_1$, we have $\one_Xf=\one_Xg$. Moreover, using $\widehat{\one_{B_1}}=\one_{B_1}$, since $B_1=\Z_p$, we obtain $\widehat g=\one_{B_1}*\widehat f$, and hence $\supp\widehat g\subseteq Y+B_1=Y_1$. Notice also that $Y_1\subseteq B_{h^{-1}}$. It is therefore enough to prove the required estimate for functions supported in $B_1$ whose Fourier transform is supported in $Y_1$.

We identify $B_1/B_h$ with $G_n=\mathbb Z/p^n\mathbb Z$ by sending $a\in G_n$ to $a+B_h$, and identify $B_{h^{-1}}/B_1$ with $G_n$ by sending $b\in G_n$ to $p^{-n}b+B_1$. Let $X_n,Y_n\subseteq G_n$ be the images of $X_h$ and $Y_1$ under these identifications.

We claim that $X_n$ and $Y_n$ are $L$-porous. There is nothing to prove if $n<L$, so suppose that $n\geq L$. Let $0\leq j\leq n-L$, and consider a residue class modulo $p^j$ in $G_n$. On the physical side this corresponds to a ball $D\subset B_1$ of radius $p^{-j}$. If $j\geq1$, then $h<p^{-j}<1$, so the porosity of $X$ gives inside $D$ a ball disjoint from $X$ of radius $\nu p^{-j}$. Since $p^{-(j+L)}\leq\nu p^{-j}$, this ball contains a ball of radius $p^{-(j+L)}$ disjoint from $X$. If $j=0$, we first choose a ball of radius $p^{-1}$ inside $B_1$. Porosity gives inside it a ball disjoint from $X$ of radius $\nu p^{-1}$, and our choice of $L$ gives $\nu p^{-1}\geq p^{-L}$.

In either case we have found, inside the original residue class, a residue class modulo $p^{j+L}$ which is disjoint from $X$. It is in fact disjoint from $X_h$. Indeed, it has radius at least $h$, and if it met a ball of radius $h$ which met $X$, then that ball would be contained in it, contradicting its disjointness from $X$. Thus $X_n$ is $L$-porous.

The same argument applies to $Y_n$. A residue class modulo $p^j$ on the Fourier side corresponds to a ball of radius $p^{n-j}$. If $j\geq1$, this radius lies strictly between $1$ and $h^{-1}$, and the porosity of $Y$ gives a disjoint ball containing a ball of radius $p^{n-j-L}$. If $j=0$, we first pass to a ball of radius $p^{n-1}$; the inequality $\nu\geq p^{1-L}$ then gives a disjoint ball of radius at least $p^{n-L}$. Since $j+L\leq n$, the ball obtained has radius at least $1$, and is therefore disjoint from $Y_1=Y+B_1$. Hence $Y_n$ is also $L$-porous.

Now let $g$ be supported in $B_1$ and suppose that $\supp\widehat g\subseteq Y_1$. Since $Y_1\subseteq B_{h^{-1}}$, the function $g$ is constant on cosets of $B_h$. Indeed, if $y\in B_h$, then $\chi(y\xi)=1$ for every $\xi\in B_{h^{-1}}$, and therefore $\widehat{g(\,\cdot+y\,)}=\widehat g$. Fourier inversion gives $g(\,\cdot+y\,)=g$. Similarly, since $g$ is supported in $B_1$, its Fourier transform is constant on cosets of $B_1$.

Define $F:G_n\to\mathbb C$ by $F(a)=p^{-n/2}g(x)$ for $x\in a+B_h$. Since each coset of $B_h$ has measure $p^{-n}$, this identification is an isometry:
$$
\|F\|_{\ell^2(G_n)}=\|g\|_{L^2(\mathbb Q_p)}.
$$
If $\xi\in p^{-n}b+B_1$, then $\chi(-x\xi)=e^{-2\pi iab/p^n}$ for $x\in a+B_h$. It follows that
\begin{align*}
\widehat g(\xi)
&=\sum_{a\in G_n}\int_{a+B_h}\chi(-x\xi)g(x)\,dx \\
&=p^{-n/2}\sum_{a\in G_n}F(a)e^{-2\pi iab/p^n}
=\mathcal F_nF(b).
\end{align*}
Thus $\supp\mathcal F_nF\subseteq Y_n$, and
\begin{equation}\label{eq:X_h-X:n-Q_p-G_n}
\|\one_{X_h}g\|_{L^2(\mathbb Q_p)}
=
\|\one_{X_n}F\|_{\ell^2(G_n)}.    
\end{equation}

Since $\mathcal F_nF$ is supported in $Y_n$ and $\mathcal F_n$ is unitary, we have
$$
\one_{X_n}F
=
\one_{X_n}\mathcal F_n^*1_{Y_n}\mathcal F_nF.
$$
Therefore,
\begin{align*}
\|\one_{X_n}F\|_{\ell^2(G_n)}
&\leq
\|\one_{X_n}\mathcal F_n^*\one_{Y_n}\|_{\ell^2\to\ell^2}
\|F\|_{\ell^2(G_n)} \\
&=
\|\one_{Y_n}\mathcal F_n \one_{X_n}\|_{\ell^2\to\ell^2}
\|F\|_{\ell^2(G_n)}.
\end{align*}
Applying Theorem~\ref{thm:finite-fup}, equation~(\ref{eq:X_h-X:n-Q_p-G_n}), and the isometry, we obtain
$$
\|\one_{X_h}g\|_{L^2(\mathbb Q_p)}
\leq
C(p,L)p^{-\beta(p,L)n}\|g\|_{L^2(\mathbb Q_p)}.
$$
Since $h=p^{-n}$, and since $L$ depends only on $p$ and $\nu$, this gives constants $C=C(p,\nu)$ and $\beta=\beta(p,\nu)>0$ such that
$$
\|\one_Xf\|_{L^2(\mathbb Q_p)}
\leq
\|\one_{X_h}g\|_{L^2(\mathbb Q_p)}
\leq
Ch^\beta\|f\|_{L^2(\mathbb Q_p)}.
$$
This proves Theorem~\ref{thm:padic-fup}.
\end{proof}

%% file: sections/proof-of-finite.tex
\section{Proof of Theorem \ref{thm:finite-fup}}

Throughout this section, we fix a prime $p$ and an integer $L\geq 1$. For each $n\geq 1$ we now define
$$
    R_n \defeq \max_{\substack{X,Y\subset G_n\\ X,Y\text{ are }L\text{-porous}}}
    \|\one_Y\F_n\one_X\|_{\ell^2(G_n)\to\ell^2(G_n)}.
$$
When $n<L$, the porosity condition is vacuous, and hence $R_n=1$. We will
prove that the sequence $(R_n)$ is submultiplicative. For self-similar
discrete Cantor sets, Dyatlov and Jin derived the corresponding
submultiplicative estimate from the exact product structure of the sets;
see~\cite{DyatlovJin2017}. General $L$-porous sets need not have such a product
structure, but their fibres and projections remain $L$-porous, which allows
the same Fourier factorization to be used here.

\subsection{Separating high and low digits}

Write $n=k+m$, with $k,m\geq 1$. Every $x,\xi\in G_n$ can be written as
$$
x = a+p^k b,\quad \xi = \eta + p^m d,
$$
where the choice of $a,d\in G_k$ and $b,\eta\in G_m$ is unique.

For a function $f$ defined on $G_n$, let
$$
f_a(b) \defeq f(a+p^kb),
$$
for any $a\in G_k$ and $b\in G_m$. A direct calculation yields the following factorization property:
\begin{equation}\label{eq:fourier-factorization}
(\F_nf)(\eta+p^md)
=p^{-k/2}\sum_{a\in G_k}
\exp\left(-\frac{2\pi iad}{p^k}\right)
\exp\left(-\frac{2\pi ia\eta}{p^n}\right)
(\F_mf_a)(\eta).
\end{equation}

Choosing the low digits of $x$ and the high digits of $\xi$ makes it so that the first exponential in \eqref{eq:fourier-factorization} is exactly the Fourier kernel of $G_k$, while the second is a unitary diagonal factor that will not affect the estimate below. We are now ready to prove submultiplicativity.

\begin{lem}\label{lem:submultiplicativity}
For all integers $m,k\geq 1$, we have
$$
R_{m+k}\leq R_m R_k.
$$
\end{lem}
\begin{proof}
Let $X,Y\subset G_{m+k}$ be $L$-porous, and suppose that $f$ is supported on $X$.
For $a\in G_k$ and $\eta\in G_m$, set
$$
X_a\defeq \{b\in G_m:a+p^kb\in X\},
\qquad
Y_\eta\defeq\{d\in G_k:\eta+p^md\in Y\}.
$$
Also let
$$
A\defeq \{a\in G_k:X_a\neq\varnothing\},
\qquad
B\defeq \{\eta\in G_m:Y_\eta\neq\varnothing\}.
$$
The sets $A$ and $Y_\eta$ are $L$-porous subsets of $G_k$, and the sets
$X_a$ and $B$ are $L$-porous subsets of $G_m$. We prove this for $A$ and
$X_a$, since the other case is analogous.

Fix $a\in G_k$. To show that $X_a\subset G_m$ is $L$-porous, let
$0\leq j\leq m-L$ and $\beta\in\mathbb Z/p^j\mathbb Z$. Then
$$
\{a+p^kb:\hspace{3pt}b\equiv\beta\pmod{p^j}\}
=
B_{k+j}(a+p^k\beta)
\subset G_{m+k}.
$$
Since $X$ is $L$-porous, there exists
$c\in\mathbb Z/p^{k+j+L}\mathbb Z$ such that
$$
c\equiv a+p^k\beta\pmod{p^{k+j}}
$$
and
$$
B_{k+j+L}(c)\cap X=\varnothing.
$$
Since $c\equiv a\pmod{p^k}$, we may write $c=a+p^k\gamma$. Then
$\gamma\equiv\beta\pmod{p^j}$ and
$$
B_{j+L}(\gamma)\cap X_a=\varnothing.
$$
Thus $X_a$ is $L$-porous.

For $A\subset G_k$, fix $0\leq j\leq k-L$ and
$\alpha\in\mathbb Z/p^j\mathbb Z$. By the $L$-porosity of $X$, there exists
$\alpha'\in\mathbb Z/p^{j+L}\mathbb Z$ such that
$$
\alpha'\equiv\alpha\pmod{p^j}
$$
and
$$
B_{j+L}(\alpha')\cap X=\varnothing.
$$
Suppose there exists $a\in A$ (which by definition means $X_a\neq \varnothing$) such that
$a\equiv\alpha'\pmod{p^{j+L}}$. Then, since $j+L\leq k$, there exists $b\in G_m$ such that
$a+p^kb\in X$, while
$$
a+p^kb\equiv\alpha'\pmod{p^{j+L}},
$$
which is in contradiction with the previous assertion we obtain from the $L$-porosity of $X$. Thus $A$ is $L$-porous.

For each $a\in G_k$, let
$$
v_a\defeq \one_B\F_mf_a.
$$
By the definition of $R_m$, we have
\begin{equation}\label{eq:first-contraction}
\sum_{a\in G_k}\|v_a\|_2^2
\leq
R_m^2\sum_{a\in G_k}\|f_a\|_2^2
=
R_m^2\|f\|_2^2.
\end{equation}
Fix $\eta\in B$. Formula~\eqref{eq:fourier-factorization} says that, as a
function of $d\in Y_\eta$, the restricted transform is
$$
\one_{Y_\eta}\F_kM_\eta\one_A
\big(v_a(\eta)\big)_{a\in G_k},
$$
where
$$
(M_\eta w)(a)
=
\exp\left(-\frac{2\pi ia\eta}{p^n}\right)w(a).
$$
The operator $M_\eta$ is diagonal and unitary, and it commutes with
$\one_A$. Therefore
$$
\|\one_{Y_\eta}\F_kM_\eta\one_A\|_{2\to2}
=
\|\one_{Y_\eta}\F_k\one_A\|_{2\to2}
\leq R_k.
$$
Hence
$$
\sum_{d\in Y_\eta}
|(\F_nf)(\eta+p^md)|^2
\leq
R_k^2\sum_{a\in G_k}|v_a(\eta)|^2.
$$
Summing over $\eta\in B$ and using \eqref{eq:first-contraction}, we obtain
$$
\|\one_Y\F_nf\|_2^2
\leq
R_k^2R_m^2\|f\|_2^2.
$$
Taking the supremum over $f$, $X$, and $Y$ proves the lemma.
\end{proof}

\subsection{Universal sampling sets}

We next recall a result on Fourier minors that we will use to obtain a strict contraction on a single block.  If $M\geq 1$, $I\subset G_M$ and $0\le s\le M$, write
$$
\chi_s(a;I)\defeq
\bigl|\{x\in I:x\equiv a\pmod{p^s}\}\bigr|,
\qquad a\in\mathbb Z/p^s\mathbb Z.
$$

\begin{thm}[Osgood--Siripuram--Wu]\label{thm:OSW}
Let $M\geq 1$, $I\subset G_M$.  The following are equivalent:
\begin{enumerate}
\item For every $0\le s\le M$ and every $a,b\in\Z/p^s\Z$,
$$
|\chi_s(a;I)-\chi_s(b;I)|\le1.
$$
\item For every $J\subset G_M$ with $|J|=|I|$, the square Fourier matrix
$$
\big(\exp(-2\pi ixy/p^M)\big)_{x\in I,y\in J}
$$
is invertible.
\end{enumerate}
\end{thm}

This is Theorem~4 of~\cite{OSW}.  In their terminology, condition (2) says that $I$ is a universal sampling set.
For our purposes, this result provides a finite substitute for unique continuation: a function with sufficiently small Fourier support is determined by its values on a suitable sampling set. This is made precise in Lemma~\ref{lem:sampling-consequence}. 

We provide a self-contained proof of the implication $(1)\implies(2)$,
which is the one used in this paper. The proof follows the argument in
Theorem~4 of~\cite{OSW}, using Lemmas~1 and~2 there.

\begin{proof}[Proof of Theorem~\ref{thm:OSW}: $(1)\implies(2)$]
Let $\omega\defeq\exp(-2\pi i/p^M)$. We choose representatives
$0\leq m_1<\cdots<m_d<p^M$ for the elements of $I$ and write
$I=\{m_1,\ldots,m_d\}$. We assume that condition~$(1)$ holds. Fix a set
$J=\{j_1,\ldots,j_d\}\subset G_M$ of cardinality $d$. We want to show
that the matrix $V=(\omega^{m_xj_y})_{x,y=1}^d$ is invertible.

We set $z_y=\omega^{j_y}$ for $1\leq y\leq d$. The matrix
$V=(z_y^{m_x})_{x,y=1}^d$ is a generalized Vandermonde matrix.
Lemma~1 of~\cite{OSW} gives a Schur polynomial $S$
such that
$$
\det V
=
\left(\prod_{1\leq u<v\leq d}(z_v-z_u)\right)
S(z_1,\ldots,z_d).
$$
We now suppose for contradiction that $\det V=0$. Then
$S(z_1,\ldots,z_d)=0$. We set
$s(x)\defeq S(x^{j_1},\ldots,x^{j_d})\in\Z[x]$. Since $s(\omega)=0$,
the cyclotomic polynomial $\Phi_{p^M}$ divides $s$. Evaluating at
$x=1$ and using $\Phi_{p^M}(1)=p$, we obtain
$p\mid S(1,\ldots,1)$.

The same lemma also gives
$$
S(1,\ldots,1)
=
\frac{
\prod_{1\leq u<v\leq d}(m_v-m_u)
}{
\prod_{1\leq u<v\leq d}(v-u)
}.
$$
We show that this integer is not divisible by $p$. Denote the numerator in the preceding formula by
$A_I\defeq\prod_{1\leq u<v\leq d}(m_v-m_u)$. For every
$1\leq s\leq M$, the condition $p^s\mid m_v-m_u$ holds precisely when
$m_u$ and $m_v$ belong to the same residue class modulo $p^s$.
Since the residue class $a$ contains $\chi_s(a;I)$ elements of $I$,
the number of such pairs is
$$
\sum_{a=0}^{p^s-1}\binom{\chi_s(a;I)}{2}.
$$
Separating the pairs according to the exact power of $p$ dividing
$m_v-m_u$, we obtain
$$
v_p(A_I)
=
\sum_{s=1}^{M-1}
s\left[
\sum_{a=0}^{p^s-1}\binom{\chi_s(a;I)}{2}
-
\sum_{a=0}^{p^{s+1}-1}\binom{\chi_{s+1}(a;I)}{2}
\right].
$$

We now compare $I$ with the set
$\widetilde I=\{0,1,\ldots,d-1\}$. For each $0\leq s\leq M$, write
$d=t_sp^s+r_s$, where $0\leq r_s<p^s$. Since the numbers
$\chi_s(a;I)$ differ by at most one and their sum is $d$, exactly
$r_s$ residue classes modulo $p^s$ contain $t_s+1$ elements of $I$,
while the remaining $p^s-r_s$ classes contain $t_s$ elements.
The same is true for $\widetilde I$. The formula above therefore gives $v_p(A_I)=v_p(A_{\widetilde I})$.

Finally, $A_{\widetilde I}=\prod_{1\leq u<v\leq d}(v-u)$, and hence
$$
v_p\bigl(S(1,\ldots,1)\bigr)
=
v_p(A_I)-v_p(A_{\widetilde I})
=
0.
$$
This contradicts $p\mid S(1,\ldots,1)$. Therefore $\det V\neq0$, and
$V$ is invertible.
\end{proof}

We only need the following special case. The sets below are called
$s$-elementary in~\cite{OSW}, and Lemma~6 there shows that they are
universal sampling sets.

\begin{cor}\label{cor:elementary-universal}
Let $M\geq 1$ and $0\le s\le M$.  Suppose that $I\subset G_M$ contains exactly one representative of every residue class modulo $p^s$.  Then $|I|=p^s$, and $I$ is a universal sampling set.
\end{cor}

\begin{proof}
If $r\le s$, then every residue class modulo $p^r$ contains exactly $p^{s-r}$ elements of $I$.  Thus
$$
\chi_r(a;I)=p^{s-r}
$$
for every $a\in\Z/p^r\Z$.  If $r>s$, then every residue class modulo $p^r$ contains at most one element of $I$, and hence
$$
\chi_r(a;I)\in\{0,1\}.
$$
The first condition in Theorem~\ref{thm:OSW} therefore holds at every level, and the result follows.
\end{proof}

We will use the following consequence.

\begin{lem}\label{lem:sampling-consequence}
Let $M\geq 1$, let $0\leq s\leq M$, and let $I\subset G_M$
contain exactly one representative of every residue class modulo $p^s$.
Suppose that $g:G_M\to\mathbb C$ satisfies
$$
g|_I=0,
\qquad
|\operatorname{supp}\F_Mg|\leq p^s.
$$
Then $g=0$.
\end{lem}

\begin{proof}
Let $\omega\defeq\exp(-2\pi i/p^M)$. Choose a set $J\subset G_M$ of cardinality $p^s$ such that
$\operatorname{supp}\F_Mg\subseteq J$. By
Corollary~\ref{cor:elementary-universal}, the set $I$ is a universal
sampling set. Therefore, by Theorem~\ref{thm:OSW}, the Fourier minor
$V_{I,J}\defeq(\omega^{ij})_{i\in I,\,j\in J}$ is invertible. Fourier
inversion gives
$$
\bigl(g(i)\bigr)_{i\in I}
=
p^{-M/2}\overline{V_{I,J}}
\bigl((\F_Mg)(j)\bigr)_{j\in J}.
$$
Since $g|_I=0$ and $\overline{V_{I,J}}$ is invertible, we have
$(\F_Mg)(j)=0$ for every $j\in J$. Together with
$\operatorname{supp}\F_Mg\subseteq J$, this gives $\F_Mg=0$, and hence
$g=0$.
\end{proof}
\subsection{A strict contraction on one block}

We now choose the length of the block. Let $q\geq 1$ be a sufficiently large integer such that
\begin{equation}\label{eq:q-choice}
(p^L-1)^q<p^{(q-1)L},
\end{equation}
and define $K\defeq 1+qL$. Such a $q$ exists because
$p^L(1-p^{-L})^q\longrightarrow0$ as $q\to\infty$.

\begin{lem}\label{lem:cardinality}
Let $K=1+qL$. If $E\subset G_K$ is $L$-porous, then
$$
|E|\leq p(p^L-1)^q<p^{K-L}.
$$
\end{lem}

\begin{proof}
At depth $1$, there are at most $p$ occupied balls. Below each such ball, every block of $L$ further levels has at most $p^L-1$ occupied descendants, since $L$-porosity excludes at least one of them. Iterating through the remaining $qL$ levels gives
$$
|E|\leq p(p^L-1)^q.
$$
Since $K-L=1+(q-1)L$, inequality~\eqref{eq:q-choice} gives
$|E|<p^{K-L}$.
\end{proof}

We first prove a strict contraction on one block. Once this is done,
submultiplicativity gives the full result by a standard argument; see,
for instance, \cite{DyatlovJin2017} in the discrete setting and
\cite{HanSchlag2020} in the continuous setting.

\begin{prop}\label{prop:block-contraction}
Let $K=1+qL$. There exists $0<\rho<1$, depending only on $p$ and $L$, such that
$$
\|\one_Y\F_K\one_X\|_{2\to2}\leq\rho
$$
for every pair of $L$-porous sets $X,Y\subset G_K$.
\end{prop}

\begin{proof}
Fix $L$-porous sets $X,Y\subset G_K$. For every $a\in\Z/p^{K-L}\Z$, porosity applied to
$B_{K-L}(a)$ gives a ball of depth $K$ contained in $B_{K-L}(a)$ and
disjoint from $Y$. Such a ball is a singleton; denote its unique point
by $i_a$. Set
$$
I=\{i_a:a\in\Z/p^{K-L}\Z\}.
$$
Then $I\subset Y^c$ contains exactly one representative of every residue
class modulo $p^{K-L}$. By Corollary~\ref{cor:elementary-universal}, it is
a universal sampling set of cardinality $p^{K-L}$.

We claim that
$$
\|\one_Y\F_K\one_X\|_{2\to2}<1.
$$
If equality held, finite-dimensionality would give a non-zero function $u$
supported on $X$ such that $\|\one_Y\F_Ku\|_2=\|u\|_2$. Since $\F_K$ is
unitary, this implies that $\F_Ku$ is supported on $Y$.

Put $g=\F_Ku$. Since $I\subset Y^c$, we have $g|_I=0$. Moreover,
$\F_Kg(x)=u(-x)$, and hence
$$
|\operatorname{supp}\F_Kg|
=
|\operatorname{supp}u|
\leq |X|
<
p^{K-L},
$$
where the final inequality follows from Lemma~\ref{lem:cardinality}.
Lemma~\ref{lem:sampling-consequence} now gives $g=0$, and therefore $u=0$,
a contradiction.

Thus the restricted norm is strictly smaller than one for each pair $X,Y$.
There are only finitely many pairs of subsets of $G_K$, so the maximum of
these norms over all $L$-porous pairs is some number $0<\rho<1$.
\end{proof}

\subsection{Completion of the proof}

We can now finish the proof of Theorem~\ref{thm:finite-fup}. Write
$n=tK+r$, where $0\leq r<K$. Using submultiplicativity over the $t$ blocks of length $K$ and bounding
the remaining block by $1$, we obtain
$R_n\leq R_K^t\leq\rho^t$.
Since $t\geq n/K-1$, this gives
$
R_n\leq\rho^{-1}\rho^{n/K}.
$
Set
$$
C=\rho^{-1},
\qquad
\beta=-\frac{\log\rho}{K\log p}>0.
$$
Then
$$
R_n\leq Cp^{-\beta n},
$$
which is exactly the assertion of Theorem~\ref{thm:finite-fup}.